\documentclass{amsart}

\usepackage{catchfilebetweentags}
\usepackage{amsfonts,amssymb}                   
\usepackage{centernot,enumerate,mathrsfs}       
\usepackage{color}                              
\usepackage[usenames,dvipsnames,svgnames,table]{xcolor} 
\usepackage{hyperref}                           
\usepackage{tikz}                               
\usepackage{cancel}
\usetikzlibrary{patterns,shapes,snakes}
\usepackage{graphicx}                           
\usepackage{listings}

\definecolor{fgcolor}{rgb}{0.345, 0.345, 0.345}
\definecolor{hlnum}{rgb}{0.686,0.059,0.569}
\definecolor{hlstr}{rgb}{0.192,0.494,0.8}
\definecolor{hlcom}{rgb}{0.678,0.584,0.686}
\definecolor{hlopt}{rgb}{0,0,0}
\definecolor{hlstd}{rgb}{0.345,0.345,0.345}
\definecolor{hlkwa}{rgb}{0.161,0.373,0.58}
\definecolor{hlkwb}{rgb}{0.69,0.353,0.396}
\definecolor{hlkwc}{rgb}{0.333,0.667,0.333}
\definecolor{hlkwd}{rgb}{0.737,0.353,0.396}
\definecolor{shadecolor}{rgb}{0.969, 0.969, 0.969}

\numberwithin{equation}{section}            
\newtheorem{thm}{Theorem}[section]          

\makeatletter
\def\footnotenonum{\xdef\@thefnmark{}\@footnotetext}
\makeatother

\DeclareMathOperator*{\argmin}{arg\,min}

\newcommand{\sqla}{\xymatrix{&\ar@{~>}[l]}}
\newcommand{\sqra}{\xymatrix{\ar@{~>}[r]&}}

\theoremstyle{plain}                             
\newtheorem{definition}[thm]{Definition}

\newtheorem{lemma}[thm]{Lemma}
\newtheorem{proposition}[thm]{Proposition}

\newtheorem{corollary}[thm]{Corollary}

\theoremstyle{definition}

\begin{document}
\title{Convex Techniques for Model Selection}
\author[Tran]{Dustin Tran}
\date{May 16, 2014}

\begin{abstract}
We develop a robust convex algorithm to select the regularization parameter in model selection. In practice this would be automated in order to save practitioners time from having to tune it manually. In particular, we implement and test the convex method for $K$-fold cross validation on ridge regression, although the same concept extends to more complex models. We then compare its performance with standard methods.
\end{abstract}
\maketitle

\section{Introduction}
Regularization has been studied extensively since the theory of ill-posed problems \cite{iv}. By adding a penalty associated to a choice of regularization parameter, one penalizes overfitted models and can achieve good generalized performance with the training data.
This has been a crucial feature in current modeling frameworks: for example, Tikhonov regularization \cite{ta}, Lasso regression \cite{ha}, smoothing splines \cite{sl}, regularization networks \cite{bm}, SVMs \cite{vv}, and LS-SVMs \cite{sj}.
SVMs in particular are characterized by dual optimization reformulations, and their solutions follow from convex programs. Standard SVMs reduce
to solving quadratic problems and LS-SVMs reduce to solving a set of linear equations.

Many general purpose methods to measure the appropriateness of a regularization parameter for given data exist: cross-validation (CV), generalized CV, Mallows's $C_p$, Minimum Description Length (MDL), Akaike Information Criterion (AIC), and Bayesian Information Criterion (BIC).
Recent interest has also been in discovering closed-form expressions of the solution path \cite{eb}, and developing homotopy methods \cite{om}.

Like SVMs, this paper takes the perspective of convex optimization---following \cite{bs}---in order to tune the regularization parameter for optimal model selection.
Classical Tikhonov regularization schemes require two steps: 1. (training) choosing a grid of fixed parameter values, find the solution for each constant regularization parameter; 2. (validating) optimize over the regularization constants and choose the model according to a model selection criterion. The approach in this paper reformulates the problem as one for constrained optimization. This allows one to compute both steps simultaneously:
minimize the validation measure subject to the training equations as constraints.

\hspace{-1em}\includegraphics[scale=0.40]{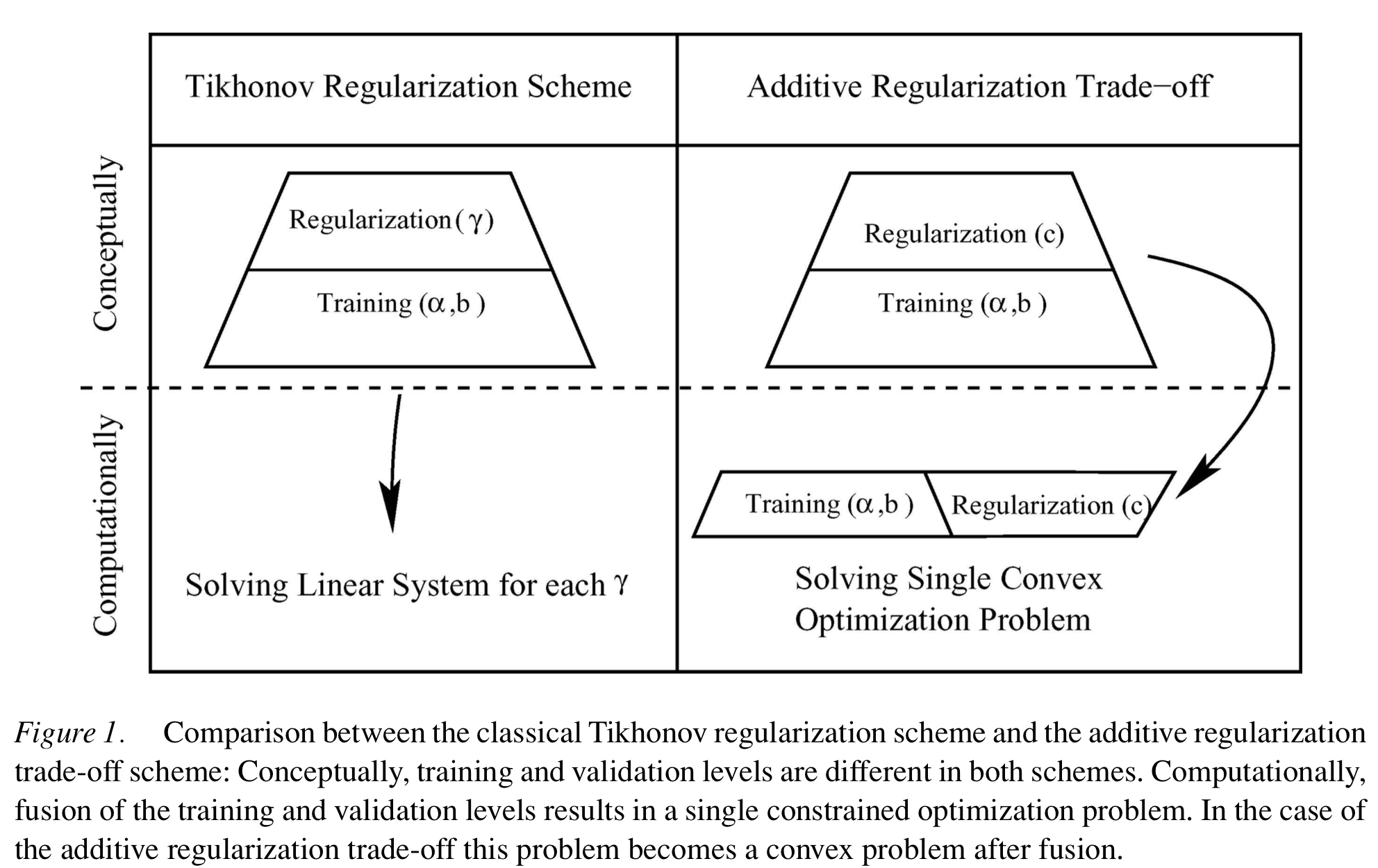}

This approach offers noticeable advantages over the ones above, of which we outline a few here:
\begin{itemize}
\item \textbf{Automation}: Practical users of machine learning tools may not be not interested in
tuning the parameter manually. This brings us closer to fully automated algorithms.
\item
\textbf{Convexity}: It is (usually) much easier to examine worst case behavior for convex sets, rather than attempting
to characterize all possible local minima.
\item
\textbf{Performance}:
The algorithmic approach of training and validating simultaneously is occasionally more efficient than general purpose optimization routines.
\end{itemize}

For this write-up, we focus only on ridge regression, although the same approach can be applied to more complex model selection problems (which may benefit more as they suffer more often from local minima).
\section{Background} 

\subsection{Ridge solution set}
We introduce the standard approach to ridge regression and prove key properties for a convex reformulation. For a more rigorous introduction, see \cite{gg}. Let $y\in\mathbb{R}^n$ and $M$ be a $n\times n$ positive semi-definite matrix.
For a fixed
$\gamma\in(0,\infty)$,
recall that Tikhonov regularization schemes of linear operators lead to the solution
$\hat{\beta}\in\mathbb{R}^n$, where the estimator $\hat{\beta}$ solves
\begin{equation}
(M + \gamma I_n)\beta = y
\end{equation}
\begin{definition}\it
For a fixed value $\gamma_0>0$, we define the \textit{ridge solution set} $S$ as the set of all solutions
$\hat{\beta}$
corresponding to a value
$\gamma\in(\gamma_0,\infty)$. That is,
\begin{equation}
S(\gamma,\beta|\gamma_0,M,y)
:= \left\{ \gamma\in(\gamma_0,\infty),~\beta\in\mathbb{R}^n | (M+\gamma I_n)\beta=y\right\}
\end{equation}
\end{definition}
The value $\gamma_0$ can be thought of as the minimal regularization parameter allowed in the solution set; this would speed up computation should the user already know a lower bound on their optimal choice of $\gamma$.

Let
$M=U\Sigma U^T$
denote the reduced singular value decomposition (SVD) of $M$, i.e., $U$ is orthogonal and $\Sigma=\mathrm{diag}(\sigma_1,\ldots,\sigma_n)$ contains all the ordered positive eigenvalues with $\sigma_1\geq\cdots\geq\sigma_n$.

\begin{lemma}\label{lemma}
The solution function $h_{M,y}(\gamma) = (M + \gamma I_n)^{-1}y$ is Lipschitz continuous.
\end{lemma}
\begin{proof}
For two values $\gamma_1, \gamma_2$ such that $\gamma_0\leq \gamma_1\leq \gamma_2 < \infty$,
\begin{align}
\|(M+\gamma_1 I_n)^{-1}y - (M + \gamma_2 I_n)^{-1}y\|_2
&= \left\|U\Big((\Sigma + \gamma_1 I_n)^{-1} - (\Sigma + \gamma_2 I_n)^{-1}\Big)U^Ty\right\|_2\\
&\leq \max_i \left|\frac{1}{\sigma_i + \gamma_1} - \frac{1}{\sigma_i + \gamma_2}\right|\|y\|_2\\
&\leq \left|\frac{1}{\sigma_n + \gamma_1} - \frac{1}{\sigma_n + \gamma_2}\right|\|y\|_2
\end{align}
Consider the function
\begin{equation}
g(x) = \frac{1}{\sigma_n + \gamma_0 + x}
\end{equation}
Then by the mean value theorem, there exists a value $c\in[\lambda_1 - \lambda_0,\lambda_2 - \lambda_0]$ such that
\begin{align}
\left|\frac{1}{\sigma_n + \lambda_1} - \frac{1}{\sigma_n + \lambda_2}\right|\|y\|_2
&\leq |g'(c)||\lambda_1 - \lambda_2|\|y\|_2\\
&\leq \frac{\|y\|_2}{(\sigma_n^2 + \lambda_0)^2}|\lambda_1 - \lambda_2|
\qedhere
\end{align}
\end{proof}

\subsection{Convex relaxation}
We now examine the the convex hull of the solution set $S(\gamma,\beta|\gamma_0,M,y)$. This will allow us to search efficiently through a convex set of hypotheses as in Section \ref{3}. We decompose the solution function into a sum of low rank matrices:
\begin{equation}
h_{M,y}(\gamma)
= (M + \gamma I_n)^{-1}y
= U(\Sigma + \gamma I_n)^{-1} U^T y
= \sum_{i=1}^n \frac{1}{\sigma_i + \gamma} U_iU_i^Ty
\end{equation}
Then define
$\lambda_i:=1/(\sigma_i + \gamma)$ for all $i=1,\ldots,n$.
The following proposition provides
linear constraints on the set
of $\lambda_i$'s following from this reparameterization.

\begin{proposition}
Let $\sigma_i':=\sigma_i+\gamma_0$
 for $i=1,\ldots,n$. Then the polytope $\mathcal{R}$ parametrized by $\Lambda=\{\lambda_1,\ldots,\lambda_n\}$ as follows
\begin{equation}\label{eqn:cvhull}
\mathcal{R}(\Lambda,\beta|\gamma_0,M,y) = \begin{cases}
U_i^T \beta=\lambda_iU_i^Ty, &\text{ for all }i=1,\ldots,n\\
0<\lambda_i\leq\frac{1}{\sigma_i'}, &\text{ for all }i=1,\ldots,n\\
\frac{\sigma_{i+1}'}{\sigma_i'}\lambda_{i+1} \leq \lambda_i \leq \lambda_{i+1}, &\text{ for all }\sigma_i' \geq \sigma_{i+1}',~i=1,\ldots,n-1
\end{cases}
\end{equation}
is convex, and moreover, forms a convex hull to $S$.
\end{proposition}
\begin{proof}
It is easy to verify the first two constraints by looking at the function $g$
defined previously, which is strictly increasing.
Set $\gamma':= \gamma-\gamma_0>0$, and $\lambda_i=1/(\sigma_i'+\gamma')$ as before. Then for all $\sigma_i'\geq\sigma_{i+1}'$,
\begin{align}
\lambda_i
&=
\Big(\frac{1}{\sigma_{i+1}' + \gamma'}\Big)/\Big(\frac{\sigma_{i+1}' + \gamma' + \sigma_{i}' - \sigma_{i+1}'}{\sigma_{i+1}' + \gamma'}\Big)\\
&= \frac{\lambda_{i+1}}{1+\lambda_{i+1}(\sigma_i' - \sigma_{i+1}')}\\
&\geq \frac{\lambda_{i+1}}{1+\frac{1}{\sigma_{i+1}'}(\sigma_i' - \sigma_{i+1}')}
= \frac{\sigma_{i+1}'}{\sigma_i'}\lambda_{i+1}
\end{align}
Hence the third inequality is also true.
Moreover, the set
$\mathcal{R}(\Lambda,\beta|\gamma_0,M,y)$
is characterized entirely by equalities and inequalities which are linear
in the unknown $\lambda_i$'s. So it forms a polytope. Then by the above result together with the convex property of polytopes, it follows that
$\mathcal{R}(\Lambda,\beta|\gamma_0,M,y)$
is a convex
relaxation to the set
$S(\gamma,\beta|\gamma_0,M,y)$.
\end{proof}

We now find the relationship between solutions in $S(\gamma,\beta|\gamma_0,M,y)$ and solutions in its convex relaxation $\mathcal{R}(\Lambda,\beta|\gamma_0,M,y)$.
\begin{proposition}
The maximal distance between a solution
in
$\mathcal R(\Lambda,\beta|\gamma_0,M,y)$
and its closest counterpart in
$S(\gamma,\beta|\gamma_0,M,y)$
is bounded by the maximum range of the inverse eigenvalue spectrum:
\begin{equation}
\min_\gamma
\|U\Sigma U^Ty - (M + \gamma I_n)^{-1}y\|_2
\leq
{{\|y\|_2}}
\max_{i> j} \left| \frac{1}{\sigma_i'} - \frac{1}{\sigma_j'}\right|
\end{equation}
\end{proposition}
\begin{proof}
Following similar steps to prove Lemma \ref{lemma}, one can show that the maximal difference between a solution
$\beta_\Gamma$ for a given $\Gamma$ and its corresponding closest $\beta_{\hat\gamma}$ is
\begin{equation}\label{eqn:min}
\min_\gamma
\| U\Lambda U^T y - (M + \hat\gamma I_n)^{-1}y\|_2
\leq
\min_\gamma \|y\|_2 \max_{i=1,\ldots,n}\left|\lambda_i - \frac{1}{\sigma_i' + \gamma}\right|
\end{equation}
For any
$\lambda_i\neq\lambda_j$, the value
\begin{equation}
\min_\gamma \max\left\{\left|\lambda_i - \frac{1}{\sigma_i' - \gamma'}\right|, \left|\lambda_j - \frac{1}{\sigma_j' - \gamma'}\right|\right\}
\end{equation}
is bounded by the worst case scenario that the
solution
$\gamma$
passes through
$\lambda_i$
or through
$\lambda_j$. Then
\begin{align}
\min_\gamma
\max_i \left| \lambda_i - \frac{1}{\sigma_i' + \gamma}\right|
&\leq \max_{i\neq j} \left| \lambda_i - \frac{1}{\sigma_i' + \gamma_j}\right|\\
&< \max_{i\neq j} \lambda_i\left| 1 - \frac{\sigma_i'}{\sigma_j'}\right|\\
&< \max_{i> j} \left| \frac{1}{\sigma_i'} - \frac{1}{\sigma_j'}\right|\label{eqn:max}
\end{align}
where $\gamma_j$ satisfies $\lambda_j= 1/(\sigma_j'+\gamma_j)$. That is, $\gamma_j = 1/\lambda_j - \sigma_j'$ which is greater than zero by construction. Hence for all $\Gamma\in\mathbb{R}^n$, there exists $\hat\gamma$ such that
\begin{equation}
\|\beta_\Gamma - \beta_{\hat\gamma}\|_2 \leq \|y\|_2\max_{i> j} \left| \frac{1}{\sigma_i'} - \frac{1}{\sigma_j'}\right|
\qedhere
\end{equation}
\end{proof}

\subsection{Reformulation of ridge regression}
Let $\mathcal{D}=\{(x_{i1},\ldots,x_{ip},y_i)\}_{i=1}^n\subset \mathbb{R}^p\times\mathbb{R}$ be a given data set. The ridge regression estimator $f(x) = x^T\beta$ with $\beta\in\mathbb{R}^p$ minimizes the regularized loss function
\begin{equation}\label{eq:beta}
\hat{\beta} = \argmin_\beta \sum_{i=1}^n \ell(y_i - x_i^T\beta) + \frac{\gamma}{2}\|\beta\|_2^2
\end{equation}
where $\ell:\mathbb{R}\to\mathbb{R}_{>0}$ is some loss function.

\begin{proposition}[KKT conditions]
Set $\ell(z):=z^2$ and fix $\gamma>0$. For $\beta\in\mathbb{R}^p$
to be the unique global minimizer of \eqref{eq:beta}, it is necessary and sufficient that $\beta$ satisfies
\begin{equation}
KKT(\beta|\gamma, \mathcal{D})~:~
(X^TX + \gamma I_p)\beta = X^Ty,
\end{equation}
where
$X$ is the $n\times p$ design matrix and $y\in\mathbb{R}^n$ is the response vector formulated from $\mathcal{D}$.
\end{proposition}

Note that we use the KKT notation in order to hint to the extension to other learning machines, which reduce to solving a similar convex optimization problem but with inequality constraints.

Let
$\mathcal{D}^v=\{(x_{i1}^v,\ldots,x_{ip}^v,y_i^v)\}_{i=1}^{n_v}\subset \mathbb{R}^p\times\mathbb{R}$
be a validation data set. The optimization problem of finding the optimal regularization parameter $\hat\beta$ with respect to a validation performance criterion $\hat{\gamma}$ can then be written as
\begin{equation}
(\hat{\beta},\hat{\gamma})
=
\argmin_{\beta,\gamma>0} \sum_{j=1}^{n_v} \ell(y_j^v - (x^v_j)^T\beta)\text{ s.t. }KKT(\beta|\gamma,\mathcal{D}) = S(\gamma,\beta|\mathcal{D})
\end{equation}
That is, we find the least squared error among all $\beta$'s in the solution set $S$, or equivalently all $\beta$'s satisfying the KKT conditions.

\section{Convex Program}\label{3} 

For the convex approach, we simply replace the non-convex solution set $S(\gamma,\beta|\mathcal{D})$ with its convex relaxation $\mathcal{R}(\Lambda,\beta|\mathcal{D})$. Then one obtains the convex optimization problem
\begin{equation}\label{cp}
(\hat{\beta},\hat{\Lambda})
=
\argmin_{\beta,\Lambda} \sum_{j=1}^{n_v} \ell(y_j^v - (x^v_j)^T\beta)\text{ s.t. } R(\Lambda,\beta|\mathcal{D})
\end{equation}
This has the immediate advantage of simultaneously training and validating (1 step); in comparison the original method requires finding a grid of points $(\hat{\beta},\hat{\gamma})$ in $S$ and then minimizing among those (2 steps). Furthermore, the convex hull is defined by $\mathcal{O}(n)$ equality/inequality constraints, whose complexity is not any higher than the original problem.

For example, \eqref{cp} can be solved with a QP solver when
$\ell(z)=z^2$ as before, or with a LP solver when $\ell(z)=|z|$ (the latter of which may be preferred for sparsenesses or feature selection).

\begin{corollary}
The convex relaxation constitutes the solution path for the modified ridge regression problem
\begin{equation}
\hat{\beta} = \argmin_\beta \sum_{i=1}^n \ell(y_i - x_i^T\beta) + \frac{1}{2}\beta^T(U\Gamma U^T)\beta
\end{equation}
where
$\Gamma=\mathrm{diag}(\gamma_1,\ldots,\gamma_p)$
and
$\gamma_i = \frac{1}{\lambda_k} - \sigma_k$
for all $i=1,\ldots,p$, and the
following inequalities hold by translating \eqref{eqn:cvhull}:
\begin{equation}
\begin{cases}
\gamma_i > 0, &\text{ for all }i=1,\ldots,p\\
\frac{\sigma_{i+1}}{\sigma_i}(\sigma_i + \gamma_i)\geq \sigma_{i+1} + \gamma_{i+1} > \sigma_i + \gamma_i, &\text{ for all }\sigma_{i+1} \geq \sigma_i,~i=1,\ldots,p-1\\
\end{cases}
\end{equation}
\end{corollary}

\subsection{Generalization to $K$-fold cross-validation}
The above applies to a single training and validation set, and we now extend it to $K$-fold CV in general.
Let $\mathcal{D}_{(k)}$ and $\mathcal{D}_{(k)}^v$ denote the set of training and validation data respectively, corresponding to the $k^{th}$ fold for $k=1,\ldots,K$: that is, they satisfy
\begin{equation}
\bigcup_k \mathcal{D}^v_{(k)} = \mathcal{D},
\qquad
\bigcap_k \mathcal{D}_{(k)} = \bigcap_k \mathcal{D}^v_{(k)} = \emptyset
\end{equation}
Let $n_{(k)} = |\mathcal{D}_{(k)}|$.
 Then in order to tune the parameter $\gamma$ according to $K$-fold CV, we have the optimization problems
\begin{align}
(\hat{\beta}_{(k)},\hat{\gamma})
&=
\argmin_{\beta_{(k)},\gamma>0} \sum_{k=1}^K \frac{1}{n-n_{(k)}} \sum_{(x_j,y_j)\in\mathcal{D}_{(k)}^v} \ell(y_j - x_j^T\beta_{(k)})\nonumber\\
&\qquad\text{ s.t. }KKT(\beta_{(k)}|\gamma,\mathcal{D}_{(k)})
 = S(\gamma,\beta_{(k)}|\mathcal{D}_{(k)})
\end{align}
for all $k=1,\ldots,K$.

Then we need only relax the KKT conditions independently for each $k$. The convex optimization problem according to a $K$-fold CV is
\begin{align}
(\hat{\beta}_{(k)},\hat{\Lambda})
&=
\argmin_{\beta_{(k)},\Lambda^k} \sum_{k=1}^K \frac{1}{n-n_{(k)}} \sum_{(x_j,y_j)\in\mathcal{D}_{(k)}^v} \ell(y_j - x_j^T\beta_{(k)})\nonumber\\
&\qquad\text{ s.t. }\mathcal{R}(\Lambda,\beta_{(k)}|\mathcal{D}_{(k)})
\end{align}
for all $k=1,\ldots,K$, each of which is solved as before, and so with $\mathcal{O}(Kn)$ constraints. Then just as in typical $K$-fold CV, we take the average of the folds
$\hat{\beta}_{avg} = \frac{1}{K}\sum_{k=1}^K \hat{\beta}_{(k)}$
as the final model.

\section{Performance}

We show two simulation studies as a benchmark in order to compare it to current methods.
The first figure below provides intuition behind the solution paths: the curve $S(\lambda,\beta|\gamma_0,M,y)$ and its convex relaxation $\mathcal{R}(\Lambda,\beta,|\gamma_0,M,y)$. In the second figure,
we simulate data with the function
$y
= \operatorname*{sinc}(x)+ \epsilon$,
where
$\epsilon\sim \mathcal N(0,1)$ are sampled i.i.d.,
$n = 50$ observations,
and
$p = 5$.

\hspace{-6.5em}\includegraphics[scale=0.50]{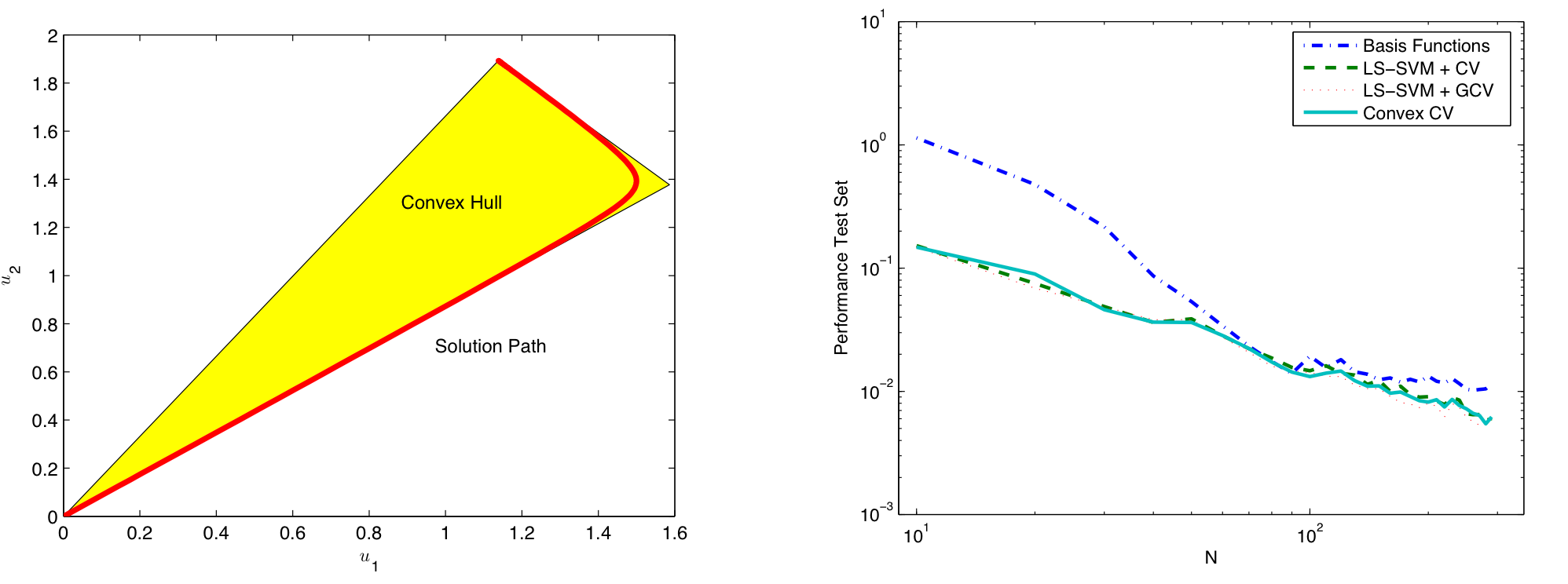}

The figure compares the performance of the method with one which uses basis functions, another with CV, and another with generaized CV (GCV).  CV and GCV are implemented with standard gradient descent, and the convex algorithm outlined here uses the interior point method.
This toy example demonstrates that the relaxation does not result in an increase in true error.

We also conduct a Monte Carlo simulation: every iteration constructs a
simulated model for a given value of
$n$
defined as
\begin{equation}
f_p(x) = w_1x^1 + \cdots + w_px^p
\end{equation}
for random values of
$w=(w_1,\ldots,w_p)\in\mathbb{R}^p$
where
$p=10$,
$w\sim\mathcal N(0,C)$,
and
$C$
is a $p\times p$
covariance matrix
with
 $\|C\|_2=1$.
A data set of size
$n$
is constructed
such that
$\mathcal D = \{(x_{i1},\ldots,x_{ip},y_i)\}_{i=1}^n$
and
\begin{equation}
y_i = f_m(x_i) + \epsilon_i,
\qquad \epsilon_i\sim\mathcal{N}(0,1)
\end{equation}
is sampled i.i.d. for all $i=1,\ldots,n$.

We compare three methods for
tuning the parameter with respect to the ordinary least squares (OLS) estimate:
10-fold CV with gradient descent
(\texttt{RR+CV})
,
generalized CV with gradient descent
(\texttt{RR+GCV})
,
and
10-fold CV criterion which applies the convex method as in (18)
(\texttt{fRR+CV}).
We run it for 20,000 iterations.

\hspace{-8em}\includegraphics[scale=0.45]{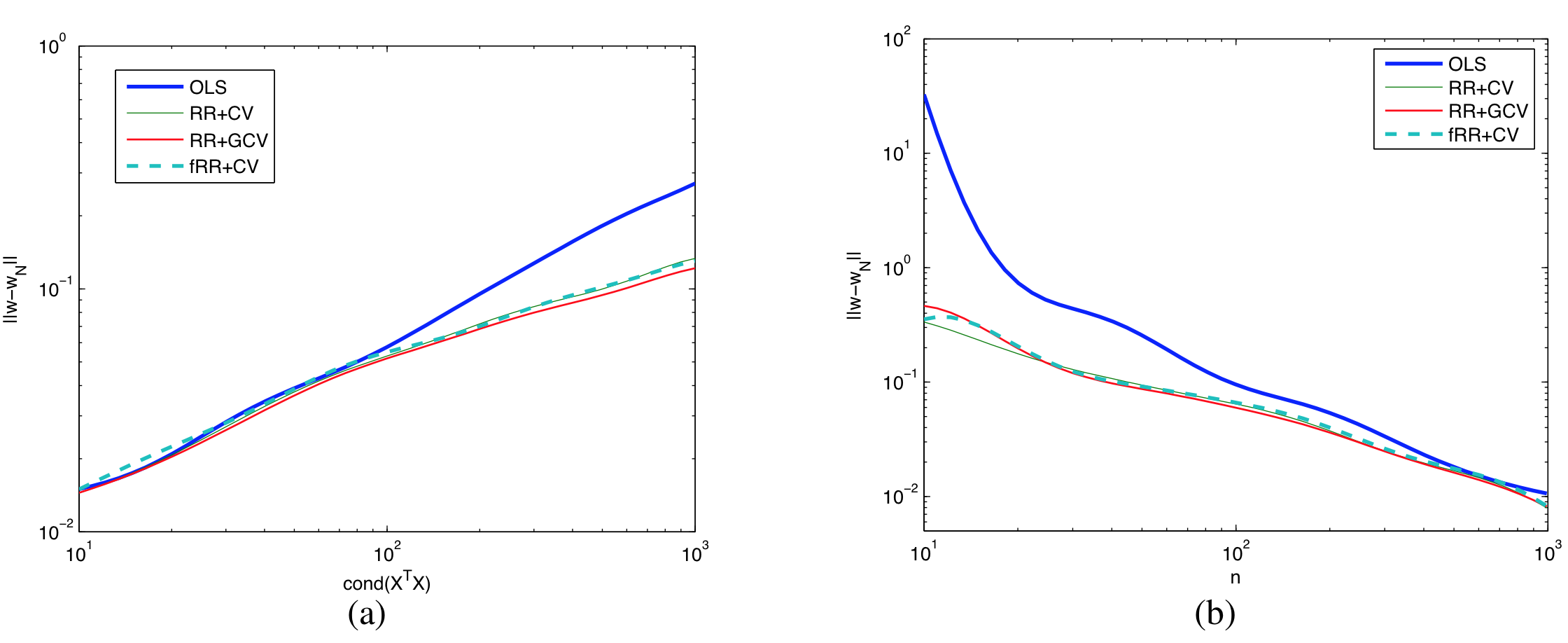}

The figure on the left compares the true error as the condition number grows; the figure on the right compares the true error as the number of observations $n$ increases.

As we can see, the method is comparable to both CV and GCV for variable changes in the data set. We expect that the OLS worsens drastically over ill-conditioned matrices, and our method compensates for that via the optimal tuning parameter which is roughly the same as CV's.

\section{Conclusion} 

Viewing the model selection problem as one in constrained optimization gives rise to a natural approach to tuning the regularization parameter. According to the simulation results, the convex program provides comparable performance to popular methods. Moreover, global optimality is guaranteed and efficient convex algorithms can be employed as suited to the modeling problem at hand. This would especially outperform general purpose techniques when there is a high number of local optima, or if finding each individual solution for fixed regularization parameter is more costly than optimizing it simultaneously with validation as we do here.

Further extensions to this framework can be used as a generic convex hull method \cite{wt},
and it can also be applied for constructing stable kernel machines, feature selection, and other possibilities related to simultaneous training and validating.



\end{document}